\documentclass[psamsfonts]{amsart}

\usepackage{amssymb,amsfonts}
\usepackage{amsthm}
\usepackage{changepage}
\usepackage{amsmath}
\usepackage{verbatim}
\usepackage[all,arc]{xy}
\usepackage{enumerate}
\usepackage{hhline}
\usepackage{mathrsfs}
\usepackage{graphicx}
\usepackage{url}
\usepackage{mathdots}
\usepackage[latin1]{inputenc}
 \usepackage{colortbl}
 \definecolor{umbra}{rgb}{0.7,0.8,0.9}
 \definecolor{grn}{rgb}{.7,1,0.7}
 \definecolor{purp}{rgb}{0.8,0.5,0.8}

\def\colCell#1#2{\multicolumn{1}{>{\columncolor{#1}}c}{#2}}

\newtheorem{thm}{Theorem}[section]
\newtheorem{cor}[thm]{Corollary}

\newtheorem{lem}[thm]{Lemma}

\newtheorem{conj}[thm]{Conjecture}

\theoremstyle{definition}
\newtheorem{defn}[thm]{Definition}

\theoremstyle{remark}

\title{Completions of $\epsilon$-Dense Partial Latin Squares}
\keywords{Latin squares, partial latin squares, H\"aggkvist, Gustavsson, improper latin squares, epsilon-dense partial latin squares}
\author{Padraic Bartlett}

\begin{document}

\begin{abstract}
A classical question in combinatorics is the following:\ given a partial latin square $P$, when can we complete $P$ to a latin square $L$?   In this paper, we investigate the class of \textbf{$\epsilon$-dense partial latin squares}:\ partial latin squares in which each symbol, row, and column contains no more than $\kern-2pt\epsilon n$-many nonblank cells.  Based on a conjecture of Nash-Williams, Daykin and H\"aggkvist conjectured that all $\frac{1}{4}$-dense partial latin squares are completable.  In this paper, we will discuss the proof methods and results used in previous attempts to resolve this conjecture, introduce a novel technique derived from a paper by Jacobson and Matthews on generating random latin squares, and use this technique to study $ \epsilon$-dense partial latin squares that contain no more than $\delta n^2$ filled cells in total.  

In this paper, we construct completions for all $ \epsilon$-dense partial latin squares containing no more than $\delta n^2$ filled cells in total, given that $\epsilon < \frac{1}{12}, \delta < \frac{ \left(1-12\epsilon\right)^{2}}{10409}$.  In particular, we show that all $9.8 \cdot 10^{-5}$-dense partial latin squares are completable.

These results improve prior work by Gustavsson, which required $\epsilon =  \delta \leq 10^{-7}$, as well as Chetwynd and H\"aggkvist, which required $\epsilon = \delta = 10^{-5}$, $n$ even and greater than $10^7$.
\end{abstract}

\maketitle

\section{Introduction/History}

A latin square of order $n$ is a $n \times n$ matrix populated by the symbols $\{1,\ldots n\}$, in which each symbol occurs at most once in every row and column; analogously, a partial latin square is such a $n \times n$ matrix in which we also allow some of these cells to be blank.  We say that a $n \times n$ partial latin square is completable if there is some way to fill in all of its blank cells to get a latin square.  This is sometimes, but not always possible, as the following two examples illustrate.
\begin{align*}
\begin{array}{|c|c|c|c|} 
\hline 1 & 2 & ~ \\
\hline~ & 3 & ~ \\
\hline ~ & ~ & 2 \\\hline
\end{array} \mapsto
\begin{array}{|c|c|c|c|} 
\hline 1 & 2 & 3 \\
\hline 2 & 3 & 1 \\
\hline 3 & 1 & 2 \\\hline
\end{array}, \textrm{ while } \begin{array}{|c|c|c|c|} 
\hline 1 & ~ & ~ \\
\hline~ &  1 & ~ \\
\hline ~ & ~ & 2 \\\hline
\end{array}
\textrm{ has no completion.}
\end{align*}

Determining the classes of partial latin squares that do admit completions is a classic genre of problems in combinatorics.  To illustrate some of these questions, we list a few classes of partial latin squares for which we have resolved this question.
\begin{itemize}
\item Any latin rectangle (i.e.\ any partial latin square $P$ where the first $k$ rows of $P$ are completely filled, while the rest are blank) can be completed (Hall \cite{Hall_1945}).
 \item  If $P$ is a partial latin square all of whose nonblank entries lie within some set of $s$ rows and $t$ columns, and $s+t \leq n$, $P$ can be completed (Ryser \cite{Ryser_1951}).
 \item If $P$ is a partial latin square with no more than $n-1$ filled cells, $P$ can be completed (Smetianuk \cite{Smetianuk_1981}).
\item  If $P$ is a $n \times n$ partial latin square with order greater than $5$ such that precisely two rows and two columns of $P$ are filled, $P$ can be completed (Buchanan \cite{Buchanan_2007}).
  \end{itemize}

In this paper, we will examine the class of \textbf{$\epsilon$-dense partial latin squares}:\ $n \times n$ partial latin squares in which each symbol, row, and column contains no more than $\epsilon n$-many nonblank cells.  This class of partial latin squares was first introduced in a paper by Daykin and H\"aggkvist \cite{Daykin_Haggkvist_1984}, where they made the following conjecture.
\begin{conj}
Any $ \frac{1}{4}$-dense partial latin square can be completed.
\end{conj}

 Gustavsson \cite{Gustavsson_1991} noted in his thesis (completed under H\"aggkvist) that this bound of $\frac{1}{4}$ was anticipated somewhat by a conjecture of Nash-Williams \cite{Nash_Williams_1970} on triangle decompositions of graphs.
\begin{conj}
Suppose that $G = (V,E)$ is a finite simple graph where each vertex has even degree, $|E|$ is a multiple of $3$, and every vertex of $G$ has degree no less than $\frac{3}{4} n$.  Then we can divide the edge set of $G$ into $|E|/3$ distinct edge-disjoint triangles.
\end{conj}
These two conjectures are linked via the following natural bijection between partial latin squares and triangulations of tripartite graphs, illustrated below.
\begin{center}
\includegraphics[width=3.5in]{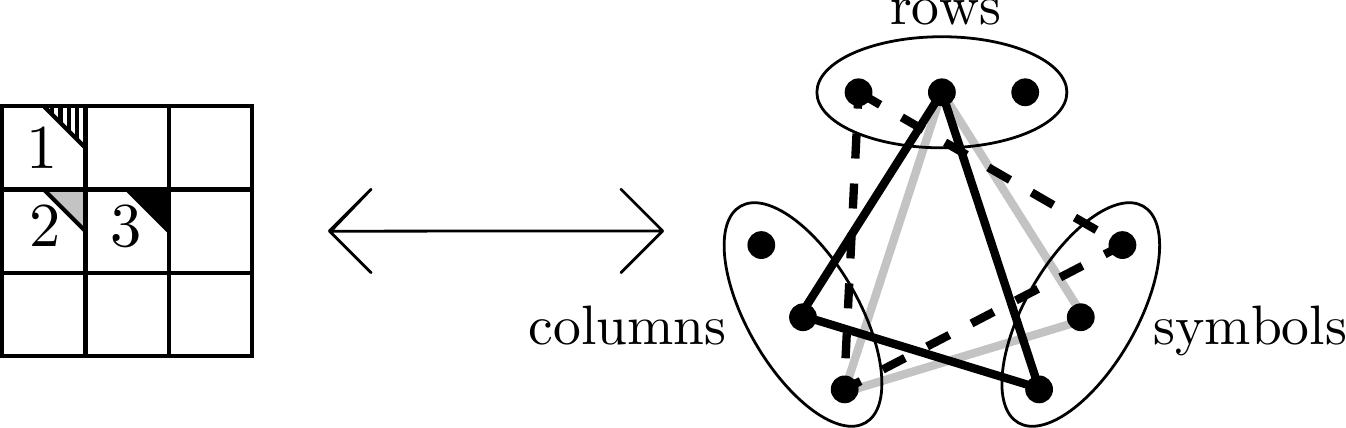}
\end{center}
In particular, given any $ \epsilon$-dense partial latin square $P$, the above transformation allows us to transform the partial latin square into a triangulated tripartite graph in which no vertex has more than $\epsilon n$-many neighbors in any one given part.  Therefore, triangulating the \textit{tripartite complement} of this graph corresponds to completing $P$ to a complete latin square; in this sense, the conjectured Nash-Williams degree bound of $\frac{3}{4}n$ suggests the Daykin-H\"aggkvist conjecture that all $ \frac{1}{4}$-dense partial latin squares are completable.

It bears noting that this bound of $\frac{1}{4}$ is tight; for any $c > 0$, there are known $ \left(\frac{1}{4} + c\right)$-dense partial latin squares (see Wanless \cite{Wanless_2002}) that cannot be completed.

In the 1984 paper where Daykin and H\"aggkvist formed this conjecture, they also proved the following weaker form of their claim.
\begin{thm}
All $\frac{1}{2^9\sqrt{n}}$-dense partial latin squares are completable, whenever $n \equiv 0 \mod 16$.
\end{thm}
This was strengthened in 1990, by Chetwynd and H\"aggkvist \cite{Chetwynd_Haggkvist_1985}, to the following theorem.
\begin{thm}
 All $ 10^{-5}$-dense partial latin squares are completable, for $n$ even and no less than $10^7$.
\end{thm}
Using Chetwynd and H\"aggkvist's result along with the above tripartite graph-partial latin square connection, Gustavsson extended this result further in his aforementioned thesis to all values of $n$, in exchange for a slightly worse bound on $\epsilon$.
\begin{thm}
All $ 10^{-7}$-dense partial latin squares are completable.
\end{thm}

The proofs in Chetwynd and H\"aggkvist's paper are difficult in parts, but the key idea behind the entire paper, \textbf{trades on latin squares}, is rather simple and elegant.  We define these trades below.
\begin{defn} 
A \textbf{trade} is a pair of partial latin squares $(P, Q)$ that satisfy the following two properties:
  \begin{itemize}
\item A cell $(i,j)$ is filled in $P$ if and only if it is filled in $Q$.
\item Any row or column in $P$ contains the same symbols as that same row or column in $Q$.
  \end{itemize}
For example, the following pair of partial latin squares form a trade.
  \begin{align*}(P,Q) = \left(~
\begin{array}{|c|c|c|c|}
\hline \colCell{umbra}{1}&  &\colCell{umbra}{3} & \\
\hline &  &  & \\
\hline \colCell{umbra}{3}&  & \colCell{umbra}{1} & \\
\hline  &  &  & \\
\hline
\end{array} , ~
\begin{array}{|c|c|c|c|}
\hline\colCell{umbra}{3}&  &\colCell{umbra}{1} & \\
\hline  &  &  & \\
\hline \colCell{umbra}{1} & &\colCell{umbra}{3} & \\
\hline &  &  & \\
\hline
\end{array}~\right)
\end{align*}
\end{defn}

In particular, Chetwynd and H\"aggkvist repeatedly use trades as a way to perform \textbf{small, local} modifications on rather large latin squares.  This is done as follows:\ suppose that $(P,Q)$ is a trade and $L$ is a completion of $P$.  Look at the array $M$ formed by taking $L$ and replacing all of $P$'s cells in $L$ with those of $Q$; by definition, this new array is still latin. 
\begin{align*} L =
\begin{array}{|c|c|c|c|}
\hline\colCell{umbra}{1} & 2 &  \colCell{umbra}{3} & 4\\
\hline 4 & 1 & 2 & 3\\
\hline \colCell{umbra}{3} & 4 & \colCell{umbra}{1} & 2\\
\hline 2 & 3 & 4 & 1\\
\hline
\end{array} \longmapsto M = 
\begin{array}{|c|c|c|c|}
\hline \colCell{umbra}{3}& 2 & \colCell{umbra}{1} & 4\\
\hline 4 & 1 & 2 & 3\\
\hline \colCell{umbra}{1}& 4 & \colCell{umbra}{3} & 2\\
\hline 2 & 3 & 4 & 1\\
\hline
\end{array}
\end{align*}
Trades of the above form, that consist of a pair of $2 \times 2$ subsquares, are particularly useful because they are the simplest trades that exist.  We call these trades $\mathbf{2} \times \mathbf{2}$ \textbf{trades} throughout this paper, and make frequent use of them.

Using trades, a rough outline of Chetwynd and H\"aggkvist's paper can be thought of as the following:
\begin{enumerate}
\item Construct  a $n \times n$ latin square $L$ in which every cell is involved in lots of well-understood $2 \times 2$ trades.
\item For every filled cell $(i,j)$ in $P$, use this structure on $L$ to find a ``simple'' trade on $L$ such that performing this trade causes $L$ and $P$ to agree at $(i,j)$.  (``Simple'' here means that it should not be too difficult to find using our given structure, nor should this trade disturb the contents of too many cells in $L$.)
\item If such trades can be found for every cell $(i,j)$, such that none of these trades overlap (i.e.\ no cell is involved in more than one trade), then it is possible to apply all of these trades simultaneously to $L$. Doing this results in a latin square that agrees with $P$ at every filled cell of $P$.
\end{enumerate}

As mentioned before, these methods work for $ 10^{-5}$-dense partial latin squares, when $n$ is even and no less than $10^7$; however, they do not seem to work on denser partial latin squares.  The main reason for this is that using small trades creates very strong \textbf{local} and \textbf{global} constraints on our partial latin square $P$ and our constructed latin square $L$, in the following ways:
\begin{itemize}
\item  In general, a $n \times n$ latin square $L$ where every cell is involved in many small well-understood trades does not always seem to \textit{exist}.  In particular, the Chetwynd and H\"aggkvist paper relies on the existence of latin squares $L$ where every cell is involved in $n/2$ distinct $2 \times 2$ trades:\ however, these squares appear to only exist in the case that $n$ is even.  (More generally, for a fixed constant $c$, latin squares $L$ where every cell is involved in $n/c$ many $2 \times 2$ trades appear to be difficult to find or construct or find whenever $n$ is odd.)
\item Moreover, if we want to follow their blueprint, we will need to find a large collection of disjoint trades on $L$; namely, one for every filled cell in $P$.  In doing this, we need to ensure that no given row (or column, or symbol) gets used ``too often'' in our trades; otherwise, it is possible that we will need to use that row at a later date to fix some other cell in $P$, and we will be unable to find a nonoverlapping trade.  This is a strong \textbf{local} constraint, as it requires us to reserve in every row/column/symbol a large swath of ``available'' cells which we have not disturbed, so that we can use their structure to construct future trades.  This also forces us to do a lot of normalization work before and during the search for these trades, in order to preserve this structure.  (This is the ``difficult'' part of Chetwynd and H\"aggkvist's proof, which otherwise is as straightforward and elegant as our earlier outline suggests.)
\item Finally, we also have a large amount of \textbf{global} constraints that we are running into.  In order to find any of these trades, we need to preserve a large amount of structure in $L$.  However, using this structure means that we need to ensure that most of $L$ still looks like the well-structured square we started with:\ consequently, each trade requires much more structure than just the cells it locally disturbs.
\end{itemize}

Given the above issues, it may seem like the technique of using trades to complete partial latin squares is a dead-end.  However, we can overcome many of these restrictions by using the concept of  \textbf{improper} latin squares and trades (introduced by Jacobson and Matthews \cite{Jacobson_Matthews_1996} in a 1996 paper on generating random latin squares.)  We define these objects below.
 \begin{defn}  A \textbf{improper latin square} $L$ is a $n \times n$ array, each cell of which contains a nonempty signed subset of the symbols $\{1, \ldots n\}$, such that the signed sum of any symbol across any row or column is 1.

A quick example:
  \begin{align*}
\begin{array}{|c|c|c|c|}
\hline 1 & 2 & 3 & 4\\
\hline 4 & 1 & 1 & 3 + 2 - 1\\
\hline 3 & 4 & 2 & 1\\
\hline 2 & 3 & 4 & 1\\
\hline
\end{array}
\end{align*}

Analogously, we can define an \textbf{partial improper latin square} as a $n \times n$ array, each cell of which contains a nonempty signed subset of the symbols $\{1, \ldots n\}$, such that the signed sum of any symbol across any row or column is either $0$ or $1$, and an \textbf{improper trade} as simply a pair of partial improper latin squares that share the same set of filled cells and the same signed symbol sums across any row or column.
 \end{defn}
Essentially, improper latin squares exist so that the following kinds of things can be considered trades.
   \begin{align*} P =
\begin{array}{|c|c|c|c|}
\hline \colCell{umbra}{a} & & \colCell{umbra}{b} & \\
\hline  &  &  & \\
\hline \colCell{umbra}{b} &  & \colCell{umbra}{c} & \\
\hline & & & \\
\hline
\end{array} \longmapsto Q = 
\begin{array}{|c|c|c|c|}
\hline \colCell{umbra}{b} & & \colCell{umbra}{a} & \\
\hline  & &  & \\
\hline \colCell{umbra}{a} & & \colCell{umbra}{b + c - a} & \\
\hline  &  & & \\
\hline
\end{array}
\end{align*}
Call these trades \textbf{improper $2 \times 2$ trades:} in practice, these will be the only improper trades that we need to use.

The main use of these improper $2 \times 2$ trades is that they let us ignore the ``local'' constraints described earlier:\ because we do not need a cell to be involved in a proper $2 \times 2$ trade in order to manipulate it, $L$ does not require any local $2 \times 2$ structure.  In particular, this lets us use latin squares $L$ of odd order, as it is not difficult to construct a latin square $L$ of odd order with a large global number of $2 \times 2$ subsquares (even though some cells will not involved in any $2 \times 2$ trades.)  We will still have the global constraints noticed earlier; in general, any system that uses only a few pre-defined types of trades seems like it will need to have some global structure to guarantee that those trades will exist.  However, just removing these local constraints gives us several advantages.
\begin{itemize}
\item Using improper trades, we can complete all partial latin squares that are $\kern-2pt 9.8 \cdot 10^{-5}$-dense, an improvement from the $ 10^{-7}$ partial latin squares of Gustavsson and the $ 10^{-5}, n\geq 10^7,$ even partial latin squares of Chetwynd and H\"aggkvist.  If we allow ourselves to examine claims that hold for larger values of $n$, we can marginally improve this to the claim that all $10^{-4}$-dense partial latin squares are completable, for $n > 1.2 \cdot 10^5.$
\item More interestingly, because we have removed these local constraints, we can now talk about completing  $ \epsilon$-dense partial latin squares that globally contain no more than $\delta n^2$-many filled cells, where $\epsilon$ and $\delta$ may not be equal.  In other words, we can now differentiate between our global and local constraints; this allows us to (in particular) massively improve our local bound $\epsilon$ at the expense of our global bound $\delta$.  For example, we can use improper trades to complete any $\mathbf{\frac{1}{13}}$\textbf{-dense} partial latin square, provided that it globally contains no more than $5.7 \cdot 10^{-7} \cdot n^2$ filled cells.  
\item In fact, given any $\epsilon \in \left[ 0, \frac{1}{12} \right)$, and any value of $\delta < \frac{ \left(1-12\epsilon\right)^{2}}{10409}$, we can show that any $ \epsilon$-dense partial latin square $P$ containing no more than $\delta n^2$ filled cells in total is completable.
\item Furthermore, because we have removed these local constraints, we can eliminate a lot of the ``normalization'' processes and techniques that Chetwynd and H\"aggkvist needed for their trades:\ consequently, these proof methods are (in some senses) easier to understand.
\end{itemize}

 The following process outlines how we will construct a completion of any such $ \epsilon$-dense partial latin square $P$ containing no more than $\delta n^2$ filled cells:
\begin{enumerate}
\item First, we will create a latin square $L$ that globally contains a large number of $2 \times 2$ subsquares.
\item Then, we will show that in any fixed row or column, it is possible to exchange the contents of ``almost any'' two cells using simple trades, provided that we have not disturbed too much of $L$'s global structure.
\item Using the above claim, we will show that given any filled cell $(i,j)$ in $P$, there is a trade that causes $L$ and $P$ to agree at this filled cell, and that does not disturb any cells at which $P$ and $L$ previously agreed.
\item By repeated applications of (3), we will turn $L$ into a completion of $P$, which is what we want.
\end{enumerate}

With our goals clearly stated and our techniques described, all that remains for us is to explicitly prove the above claims.
\section{The Proof}

We begin by creating latin squares with ``lots'' of well-understood $2 \times 2$ subsquares.
\begin{lem}\label{lem1}
For any $k$ , there is a $2k \times 2k$ latin square $L$ of the form $\begin{array}{|c|c|} \hline A & B \\ \hline B^T & A^T \\ \hline\end{array}$, with the following property: if there are two cells $(i,j), (i',j')$ in opposite quadrants containing the same symbol, then there is a $2 \times 2$ trade that exchanges the contents of these two symbols.

Furthermore, there is a way to extend this construction to an $n \times n$ odd-order square, in such a way that preserves this property at all but $3n + 7$ cells in the new odd-order square.
\end{lem}
\begin{proof}
For even values of $n$, we can simply use the following construction used by Chetwynd and H\"aggkvist in their proof.
\begin{align*}
L = \begin{array}{|c|c|c|c||c|c|c|c|}
\hline 1 & 2 & 3 & 4 & 5 & 6 & 7 & 8 \\
\hline 4 & 1 & 2 & 3 & 8 & 5 & 6 & 7 \\
\hline 3 & 4 & 1 & 2 & 7 & 8 & 5 & 6 \\
\hline 2 & 3 & 4 & 1 & 6 & 7 & 8 & 5 \\
\hhline{|=|=|=|=||=|=|=|=|} 
	   5 & 8 & 7 & 6 & 1 & 4 & 3 & 2 \\
\hline 6 & 5 & 8 & 7 & 2 & 1 & 4 & 3 \\
\hline 7 & 6 & 5 & 8 & 3 & 2 & 1 & 4  \\
\hline 8 & 7 & 6 & 5 & 4 & 3 & 2 & 1  \\
\hline
\end{array}
\end{align*}

In general, their construction is the following:\ if we set $A$ as the $k \times k$ circulant matrix on symbols $\{1, \ldots k\}$ and $B$ as the $k \times k$ circulant matrix on symbols $\{k+1, \ldots 2k\}$, we can define $L$ as the $n \times n$ latin square given by $\begin{array}{|c|c|} \hline A & B \\ \hline B^T & A^T \\ \hline\end{array}$.  An example for $n = 8$ is provided above.

This latin square $L$, as noted by Chetwynd and H\"aggkvist, has the following property:\ every cell in $L$ is involved in precisely $n/2$ distinct $2 \times 2$ subsquares.  To see this, notice that another way to describe $L$ is as the following.
\begin{align*}
L(i,j) = \left\{ \begin{array}{ll}
j-i + 1 \mod k & \textrm{for } i, j \leq k, \\
i-j + 1 \mod k   & \textrm{for } i >  k, j > k, \\
(j-i + 1 \mod k) + k  & \textrm{for } i \leq k, j > k, \\
 (i-j + 1 \mod k) + k  & \textrm{for }  i > k, j \leq k, \\
\end{array} \right.
\end{align*}
Take any cell $(i,j)$ in our latin square $L$.  Pick any other cell $(i,y)$ from the same row as $(i,j)$, but from the opposite quadrant.  Pick the cell $(x,j)$ so that it has the same symbol as the symbol in $(i,y)$:\ then we have that 
\begin{align*}
y-i + 1 \equiv x-j + 1 \mod k.
\end{align*}
This implies that
\begin{align*}
y-x + 1 \equiv i-j + 1 \mod k;
\end{align*}
i.e that the symbols in cells $(i,j)$ and $(x,y)$ are the same.  Therefore, any cell $(i,j)$ in our latin square $L$ is involved in precisely $n/2$-many $2 \times 2$ subsquares:\ one for every cell in the same row and opposite quadrant.

For $n = 4k +1$ for some $k$, we can augment Chetwynd and H\"aggkvist's construction as follows.  First, use the construction above to construct a $4k \times 4k$ latin square $L$. Now, consider the transversal of $L$ consisting of the following cells.
\begin{center}
$(1,1), (2,3), (3,5), (4,7) \ldots (k, 2k-1), $\\
$(k+1,2k+1), (k+2,2k+3),  \ldots (2k, 4k-1), $\\
$(2k+1, 2k+2), (2k+2,2k+4), \ldots (3k, 4k)$\\
$(3k+1, 2), (3k+2, 4), \ldots (4k,2k)$.\\
\end{center}
\footnotesize
\begin{align*}
\begin{array}{|c|c|c|c|c|c||c|c|c|c|c|c|}
\hline \colCell{umbra}{1} & 2 & 3 & 4 & 5 & 6 &A & B & C & D & E & F \\
\hline 6 & 1 & \colCell{umbra}{2} & 3 & 4 & 5 &F & A & B & C & D & E \\
\hline 5 & 6 & 1 & 2 & \colCell{umbra}{3} & 4 &E & F & A & B & C & D \\
\hline 4 & 5 & 6 & 1 & 2 & 3 & \colCell{umbra}{D} & E & F & A & B & C \\
\hline 3 & 4 & 5 & 6 & 1 & 2 &C & D & \colCell{umbra}{E} & F & A & B \\
\hline 2 & 3 & 4 & 5 & 6 & 1 &B & C & D & E & \colCell{umbra}{F} & A \\
 \hhline{>{\doublerulesepcolor{white}}|=|=|=|=|=|=|t|=|=|=|=|=|=|}
         A & F & E & D & C & B &1 & \colCell{umbra}{6} & 5 & 4 & 3 & 2 \\
\hline B & A & F & E & D & C &2 & 1 & 6 & \colCell{umbra}{5} & 4 & 3 \\
\hline C & B & A & F & E & D &3 & 2 & 1 & 6 & 5 & \colCell{umbra}{4} \\
\hline D & \colCell{umbra}{C} & B & A & F & E &4 & 3 & 2 & 1 & 6 & 5 \\
\hline E & D & C & \colCell{umbra}{B} & A & F &5 & 4 & 3 & 2 & 1 & 6 \\
\hline F & E & D & C & B & \colCell{umbra}{A} &6 & 5 & 4 & 3 & 2 & 1 \\
\hline
\end{array}\\
\textrm{(The above transversal in a }13 \times 13\textrm{ latin square.)}~\quad
\end{align*}
\normalsize

Using this transversal, turn $L$ into a $4k+1 \times 4k+1$ latin square $L'$ via the following  construction:\ take $L$, and augment it by adding a new blank row and column.  Fill each cell in this blank row with the corresponding element of our transversal that lies in the same column as it; similarly, fill each cell in this blank column with the corresponding transversal cell that is in the same row.  Finally, replace the symbols in every cell in our transversal (as well as the blank cell at the intersection of our new row and column) with the symbol $4k+1$.  This creates a $n \times n$ latin square, such that all but $3n-2$ cells are involved in precisely $(n/2) - 2$ distinct $2\times 2$ subsquares.

For $n = 4k-1$, things are slightly more difficult.  While we can use our earlier construction to create a $4k-2 \times 4k-2$ latin square, the resulting square does not have a transversal.  However, we can use our $2 \times 2$ substructure to slightly modify this square so that it will have a transversal, and then proceed as before.  We outline the process for creating this transversal below:
\begin{itemize}
\item First, use the Chetwynd and H\"aggkvist construction to create a $(4k-2) \times (4k-2)$ latin square $L$.
\item In our discussion earlier, we noted that for any pair of cells $(i,j), (i,k)$ in the same row but from different quadrants, there is a $2 \times 2$ subsquare that contains those two elements.  Take the $2 \times 2$ subsquare corresponding to the cells containing $2$ and $4k-2$ in the last row, and perform the $2 \times 2$ trade corresponding to this subsquare.
\item Similarly, take the $2 \times 2$ subsquare corresponding to the cells containing $2k-1$ and $4k-2$ in the far-right column, and perform the $2 \times 2$ trade corresponding to this subsquare.
\item With these two trades completed, look at the four cells determined by the last two rows and columns of our latin square.  They now form a $2 \times 2$ subsquare of the form $\begin{array}{|c|c|}\hline 1 & 4k-2 \\ \hline 4k-2 & 1\\  \hline \end{array}$.  Perform the trade corresponding to this $2 \times 2$ subsquare.
\end{itemize}

Once this is done, we can find a transversal by simply taking the cells
\begin{center}
$(1,1), (2,3), (3,5), (4,7) \ldots (k, 2k-1), $\\
$(k+1,2k), (k+2,2k+2),  \ldots (2k-1, 4k-4), $\\
$(2k, 2k+1), (2k+2,2k+3), \ldots (3k - 2, 4k - 3),$\\
$(3k - 1, 2), (3k, 4), \ldots (4k - 3,2k - 2),$\\
$(4k-2, 4k-2).$
\end{center}
\footnotesize
\begin{align*}
\begin{array}{|c|c|c|c|c|c|c||c|c|c|c|c|c|c|}
\hline \colCell{umbra}{1} & 2 & 3 & 4 & 5 & 6 & G      &      A & B & C & D & E & F & 7 \\
\hline 7 & 1 & \colCell{umbra}{2} & 3 & 4 & 5 & 6      &      G & A & B & C & D & E & F \\
\hline 6 & 7 & 1 & 2 & \colCell{umbra}{3} & 4 & 5      &      F & G & A & B & C & D & E \\
\hline 5 & 6 & 7 & 1 & 2 & 3 & \colCell{umbra}{4}      &      E & F & G & A & B & C & D \\
\hline 4 & 5 & 6 & 7 & 1 & 2 & 3      &      \colCell{umbra}{D} & E & F & G & A & B & C \\
\hline 3 & 4 & 5 & 6 & 7 & 1 & 2      &      C & D &\colCell{umbra}{E} & F & G & A & B \\
\hline G & 3 & 4 & 5 & 6 & 7 & 1      &      B & C & D & E & \colCell{umbra}{F} & 2 & A \\
 \hhline{>{\doublerulesepcolor{white}}|=|=|=|=|=|=|=|t|=|=|=|=|=|=|=|}
\hline A & G & F & E & D & C & B      &      1 & \colCell{umbra}{7} & 6 & 5 & 4 & 3 & 2 \\
\hline B & A & G & F & E & D & C      &      2 & 1 & 7 & \colCell{umbra}{6} & 5 & 4 & 3 \\
\hline C & B & A & G & F & E & D      &      3 & 2 & 1 & 7 & 6 & \colCell{umbra}{5} & 4 \\
\hline D & \colCell{umbra}{C} & B & A & G & F & E      &      4 & 3 & 2 & 1 & 7 & 6 & 5 \\
\hline E & D & C & \colCell{umbra}{B} & A & G & F      &      5 & 4 & 3 & 2 & 1 & 7 & 6 \\
\hline F & E & D & C & B & \colCell{umbra}{A} & 7      &      6 & 5 & 4 & 3 & 2 & G & 1 \\
\hline 2 & F & E & D & C & B & A      &      7 & 6 & 5 & 4 & 3 & 1 & \colCell{umbra}{G} \\
\hline
\end{array}\\
\qquad\textrm{(The above transversal in a }14 \times 14\textrm{ latin square.)} \qquad \quad
\end{align*}
\normalsize
Using this transversal, we can extend our latin square to a $4k-1 \times 4k-1$ latin square, using the same methods as in the $n = 4k+1$ case.  This completes the first step of our outline.
\end{proof}

\indent Our next lemma, roughly speaking, claims the following:\ given any latin square generated by Lemma $\ref{lem1}$, we can pick any row and exchange the contents of ``many'' pairs of cells within that row, without disturbing other cells in that row, more than 16 cells in our entire latin square, or some prescribed small set of symbols that we'd like to avoid disturbing in general.  Furthermore, the following result claims that we can do this \textit{repeatedly}:\ i.e.\ that we can apply this result not just to latin squares generated by Lemma $\ref{lem1}$, but to latin squares generated by Lemma $\ref{lem1}$ that have had the contents of up to $k n^2$ many cells disturbed by such trades, for some constant $k$ that we will determine later.  We state this claim formally below.
\begin{lem}\label{lem2}
Initially, let $L$ be one of the $n\times n$ latin squares constructed by Lemma $\ref{lem1}$, and $P$ be an $ \epsilon$-dense partial latin square.  Perform some sequence of trades on $L$, and suppose that after these trades are completed that the following holds:\ no more than $kn^2$ of $L$'s cells either have had their contents altered via such trades, or were part of the $3n + 7$ potentially-disturbed cells that were disturbed in the execution of Lemma $\ref{lem1}$.  As well, fix any set $\{t_1, \ldots t_a\}$ of symbols.

Fix any positive constant $d > 0$.  Then, for any row $r_1$ of $L$ and all but 
\begin{itemize}
\item $2\frac{k}{d}n + \epsilon n + a$ choices of $c_1$, and
\item $4\frac{k}{d}n + 2d n + 3\epsilon n + a+ 1$ choices of $c_2,$
\end{itemize}
there is a trade on $L$ that
\begin{itemize}
\item does not change any cells on which $P$ and $L$ currently agree,
\item changes the contents of at most 16 cells of $L$,
\item does not use any of the symbols $\{t_1, \ldots t_a\}$, and
\item swaps the symbols in the cells $(r_1, c_1)$ and $(r_2, c_2)$,
\end{itemize}
as long as the following two equations hold:
\begin{itemize}
\item $3 \leq n - 4\frac{k}{d}n - 6d n  - 6 \epsilon n - 3a, \textrm{ and}$
\item $12 \leq n - 12d n -  12\epsilon n - 4a.$
\end{itemize}
\end{lem}
\begin{proof}

Call a row, column, or symbol in $L$ $d$-\textbf{overloaded} (or just overloaded, for short) if $>\kern-2pt d n$ of the entries in this row/column/symbol have had their contents changed by the trades we have performed thus far on $L$, counting the $3n+7$ possibly-disturbed cells from $L$ 's construction as such changed cells.  Note that no more than $\frac{kn^2}{d n} = \frac{k}{d} n$ rows, columns, or symbols are overloaded.

Intuitively, overloaded rows are going to be ``difficult'' to work with:\ because most of the structure in our latin square no longer exists within that row, we will have relatively few ways to reliably manipulate the cells in that row.  Conversely, if some row is not overloaded, we know that the contents of most of the cells within this row have not been disturbed; in theory, this will make manipulating this row much easier, as we will have access to a lot of the structure we have built into our latin square $L$.  (Similar comments apply to overloaded symbols and columns.)

With these comments as our motivation, we begin constructing our trade.  Fix some row $r_1$:\ we want to show that for most pairs of cells within this row, there is a trade which exchanges their contents without disturbing many other cells in $L$.  Naively, we might hope that for most pairs of cells in our row, we can find the following trade:
\begin{center}
\includegraphics[width=2.7in]{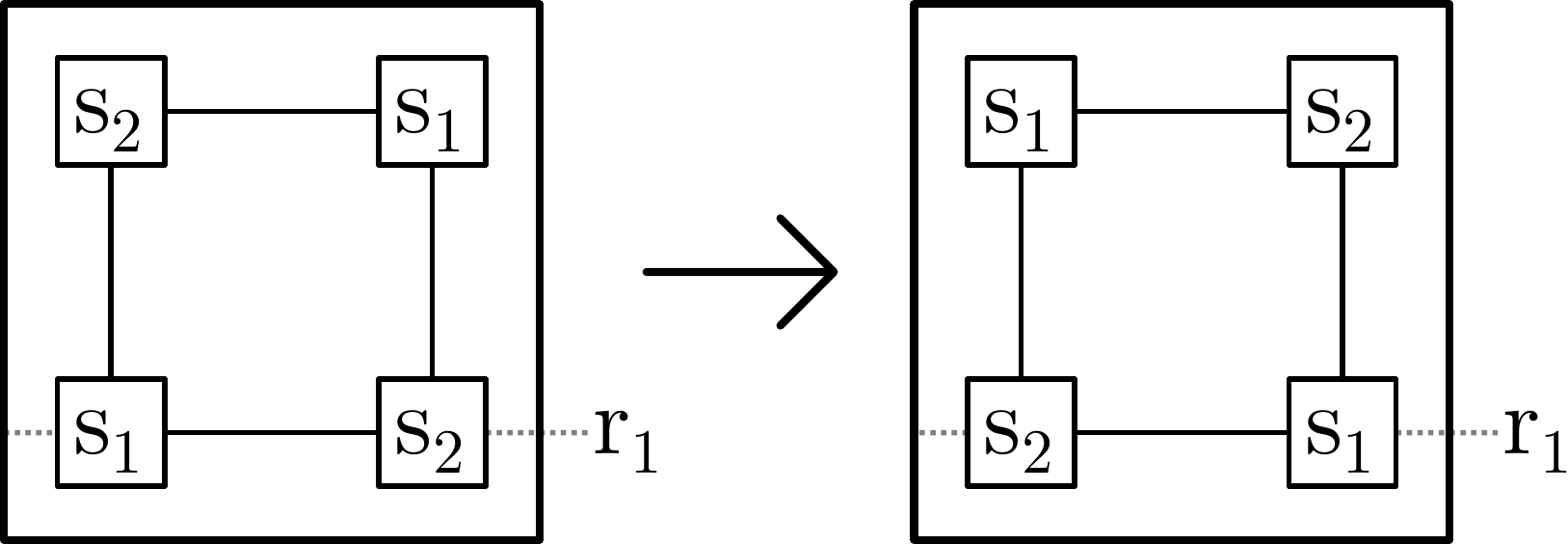}
\end{center}
If this situation occurs, we can simply perform the $2 \times 2$ trade illustrated above to swap the two cells containing $s_1$ and $s_2$.  The issue, however,  is that this situation may never come up:\ if $r_1$ is an overloaded row, for example, it is entirely possible that \textbf{none} of its elements are involved in \textbf{any} $2 \times 2$ subsquares.  To fix this, we use the technique of \textbf{improper trades}:\ specifically, we will choose some other row $r_2$, and perform the improper trade
\begin{center}
\includegraphics[width=3.2in]{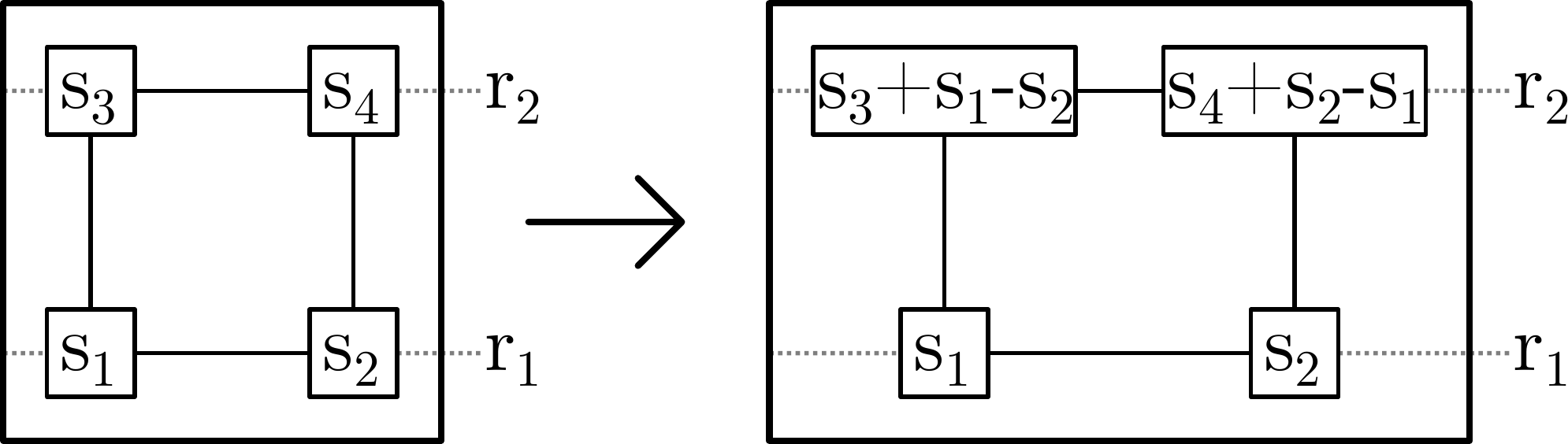}
\end{center}

This accomplishes our original goal of exchanging these two elements in $r_1$; however, we may now have an improper latin square, if either $s_3 \neq s_2$ or $s_4 \neq s_1$.  The aim of this lemma is to construct a \textbf{proper} trade on our latin square:\ therefore, we need a way to augment this trade so that it becomes a proper one.  This is not too difficult to do:\ by repeatedly stringing together improper $2 \times 2$ trades that use nonoverloaded rows/columns/symbols wherever possible, and using cells that have not been disturbed by earlier trades where possible, we can augment the improper trade above to one of two possible proper trades. We illustrate the first of these below:
\begin{center}
\includegraphics[width=3.9in]{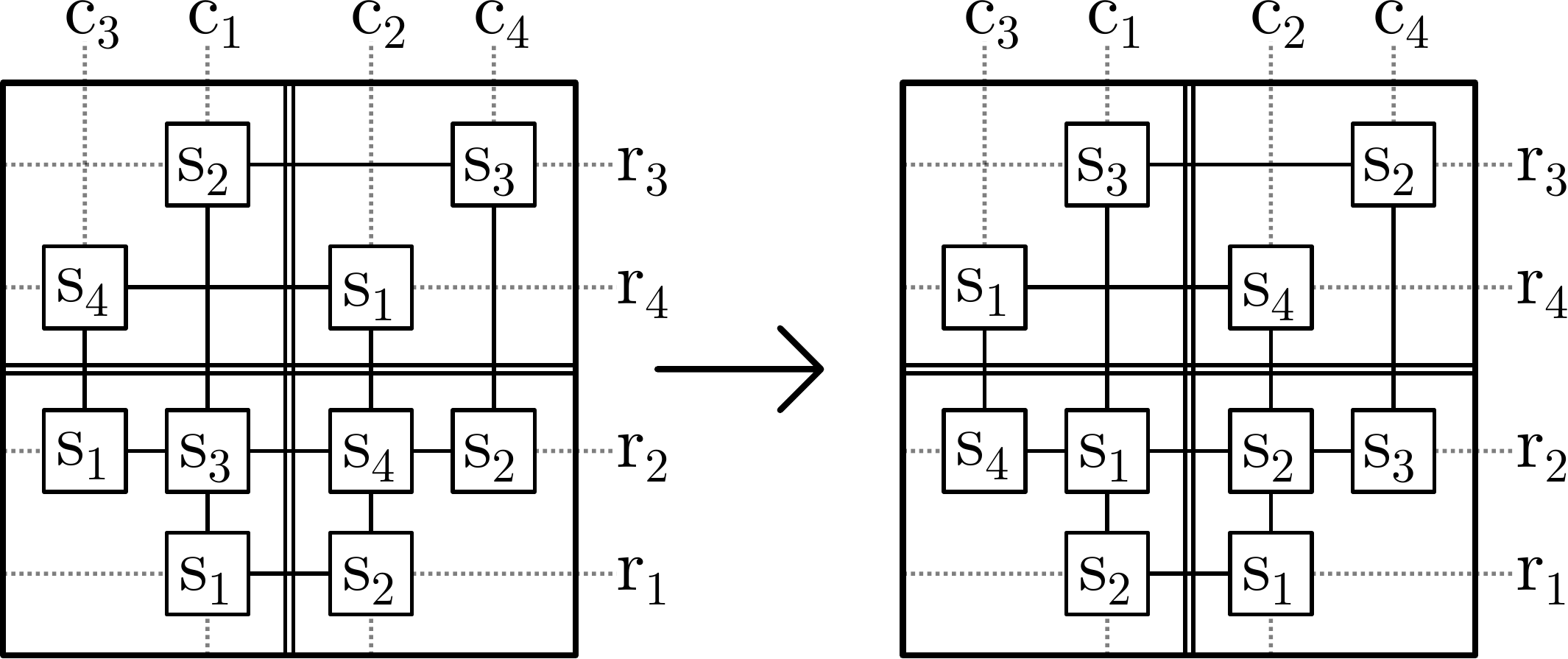}
\end{center}

The simpler trade of the two, illustrated above, occurs in the following situation.  Look at the two cells $(r_1, c_1)$ and $(r_1, c_2)$, and suppose that the symbols in both of these cells are not overloaded.  For each of these two cells, there are two possibilities:\ either the cell $(r_1, c_i)$ is still in the same quadrant that the symbol $s_i$ started in, or it has been permuted to the other quadrant.  The trade we have just drawn above occurs in the situation where either both of these cells are still in the same quadrants that their corresponding symbols started in, or when neither of these cells are in the same quadrants that their corresponding symbols start in.  (This condition is equivalent to asking that rows $r_3, r_4$ are both in the same half of our square, which is necessary for our choice of $r_2$.)

We show that this trade can always be found using the following heuristic:\ we will choose the rows/columns/symbols involved in this trade one by one, choosing each so that as many of the variables determined by our choice are involved in nonoverloaded rows/columns/symbols as possible.  Furthermore, we will also attempt to insure that as many cells as possible in our trade have never had their contents disturbed by either earlier trades on $L$ or from being part of the potential $3n+7$ disturbed cells in $L$'s construction.  Finally, we will also make sure that our choices never involve cells where $L$ and $P$ currently agree.

Start by choosing $s_1$ such that the following properties hold:
\begin{itemize}
\item The symbol $s_1$ is not overloaded.  As well, if $c_1$ is the column such that $(r_1, c_1)$ contains $s_1$, the column $c_1$ should also not be overloaded.  This eliminates at most $2\frac{k}{d}n$ choices.
\item The cell $(r_1, c_1)$ is not one at which $P$ and $L$ currently agree.  This eliminates at most $\epsilon n$ choices.
\item The symbol in the cell $(r_1, c_1)$ is not one of $\{t_1, \ldots t_a\}$.  This eliminates at most $a$ choices.
\end{itemize}
Therefore, we have 
\begin{align*}
n - 2\frac{k}{d }n - \epsilon n - a
\end{align*}
choices for $s_1$, as claimed.

We choose the second symbol $s_2$ so that a similar set of properties hold:
\begin{itemize}
\item The symbol $s_2$ should not be the same as $s_1$, nor should it be one of $\{t_1, \ldots t_a\}$.  This eliminates at most $a+1$ choices.
\item The cell in column $c_1$ containing $s_2$ has not been used in any earlier trades; as well, the cell containing the symbol $s_1$ in the column $c_2$ has not been used by any earlier trades.  Because neither $c_1$ nor $s_1$ are overloaded, we know that at most $dn$ entries in either of these objects have been used in previous trades.  Therefore, this restriction eliminates at most $2dn$ choices.
\item The rows $r_3, r_4$ containing these two undisturbed cells are not overloaded.  As well, we ask that neither the symbol $s_2$ nor the column $c_2$ is overloaded.  This eliminates at most $4\frac{k}{d}n$ choices.
\item Neither the cell $(r_1, c_2)$ nor the two undisturbed cells are cells at which $P$ and $L$ currently agree.  This eliminates at most $3\epsilon n$ more choices.
\end{itemize}
This leaves
\begin{align*}
n - 4\frac{k}{d}n - 2d n -  3\epsilon n - a - 1
\end{align*}
choices for $s_2$, again as claimed.

We have one final choice to make here:\ the row $r_2$.  Observe that because the cells $(r_1, c_i)$ are either both in the same quadrant that the cell containing $s_i$ in row $r_1$ started in, or both permuted to the quadrants they were not in, the cells $(r_3, c_1)$ and $(r_4, c_2)$ are both either in the top half or both in the bottom half of our latin square.  Using this observation, we can choose $r_2$ so that the following conditions hold:

\begin{itemize}
\item The row $r_2$ is in the opposite half from the rows $r_3, r_4$, and is not $r_1$.  This eliminates at most $\lceil n/2 \rceil + 1$  choices.
\item None of the cells $(r_2, c_1)$, $(r_2, c_2)$, $(r_2, c_3)$, $(r_2, c_4)$,  $(r_3, c_4)$, or $(r_4, c_3)$  have been used in prior trades.  Because neither the columns $c_1$, $c_2$, nor the symbols $s_1, s_2$, nor the rows $r_3, r_4$ are overloaded, this restriction eliminates at most $6d n$ choices.  
\item None of these cells are cells at which $P$ and $L$ currently agree.  This eliminates at most $6\epsilon n$ choices.
\item Neither $s_3$ or $s_4$ are equal to any of the symbols $\{t_1, \ldots t_a\}$.  This eliminates at most $2a$ choices.
\end{itemize}
This leaves
\begin{align*}
\left\lfloor\frac{n}{2}\right\rfloor - 6d n -  6\epsilon n - 2a - 1
\end{align*}
choices for the row $r_2$. 

Notice that because $r_2$ has been chosen to be from the opposite half of $r_3, r_4$, and none of these cells nor the earlier have been disturbed by earlier trades, we know that the symbol $s_3$ is in the cell $(r_3, c_4)$ and the symbol $s_4$ is in the cell $(r_4, c_3)$:\ this is because of $L$'s previously-discussed ``many $2 \times 2$ subsquares'' structure.  Therefore, whenever we can make all of these choices,  we have constructed the trade that we claimed was possible.  

The slightly more complex trade that we have to consider is when the above choice of $r_2$ is impossible:\ i.e.\ when one of the rows $r_3,r_4$ is in the top half and the other is in the bottom half of $L$.  We deal with this obstruction via the following trade, which (again) was constructed by repeatedly applying improper $2 \times 2$ trades.
\begin{center}
\includegraphics[width=5in]{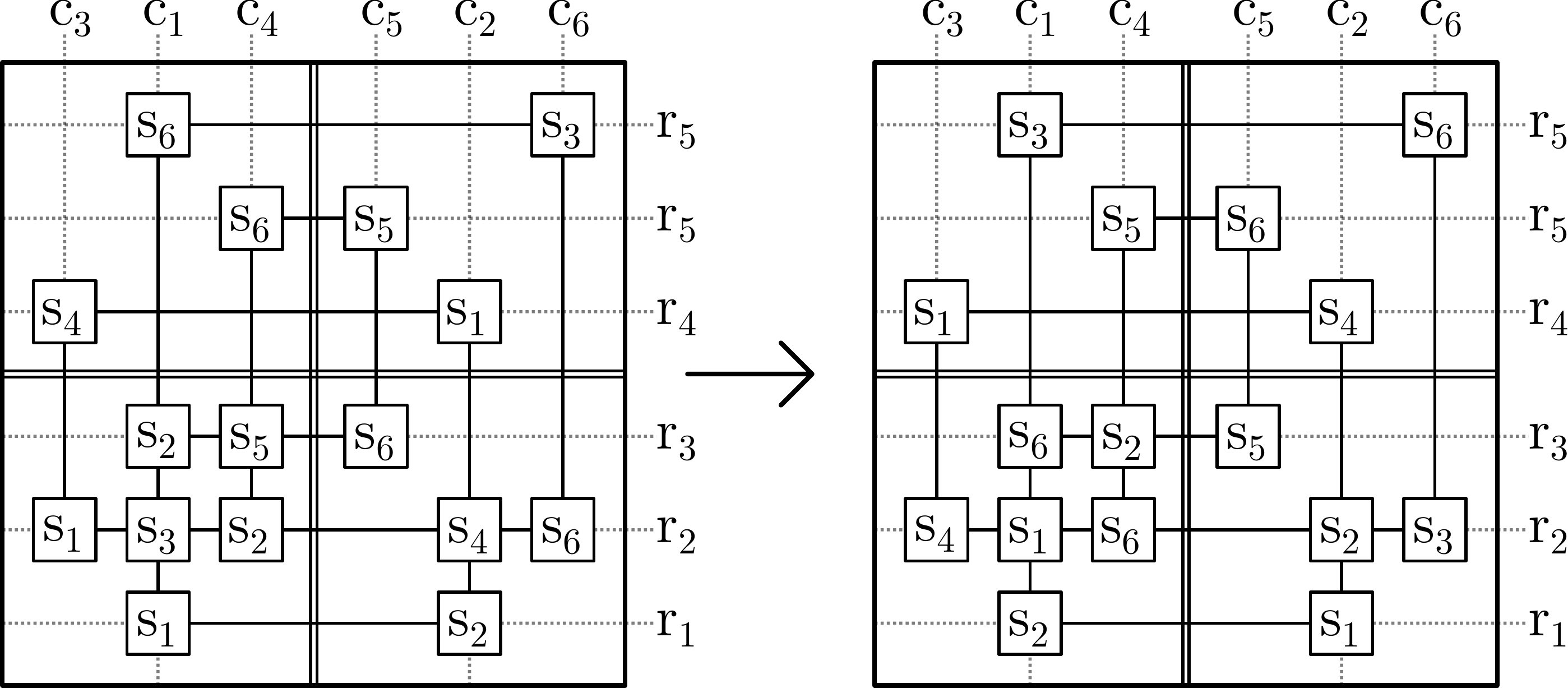}
\end{center}

Choose $s_1, s_2$ exactly as before.  We will now choose values of $r_2$ and $s_6$ so that the trade illustrated above exists.

We start out with $n$ possible choices of $r_2$:\ for any such choice, either $r_2$ and $r_3$ will be in different halves, or $r_2$ and $r_4$ will be in different halves of our latin square.  Assume without loss of generality that $r_2$ and $r_4$ are in different halves of $L$:\ the case where $r_2$ and $r_3$ are in different halves is identical.  In this case, pick $r_2$ such that the following properties hold:
\begin{itemize}
\item  None of the cells $(r_2, c_1)$, $(r_2, c_2)$, $(r_2, c_3)$, $(r_2, c_4)$,  $(r_4, c_3)$, or $(r_3, c_4)$  have been disturbed in prior trades.  Because neither the columns $c_1$, $c_2$, nor the symbols $s_1, s_2$, nor the rows $r_3, r_4$ are overloaded, this restriction eliminates at most $6d n$ choices. (Notice that because $r_2$ and $r_4$ are in different halves, we know that the same symbol $s_4$ is in $(r_4, c_3)$ and $(r_2, c_2)$; again, this is caused by $L$'s well-understood ``many $2\times 2$ subsquares'' structure.  However, unlike our earlier case, we cannot make a similar assumption for the cell $(r_3, c_4)$.)
\item None of the following are overloaded:\ the symbol $s_3$, the symbol $s_5$, the row $r_2$, or the column $c_4$.   Furthermore, none of the cells determined by these choices are in use in our trade thus far.  This eliminates at most $4\frac{k}{d}n + 2$ choices.
\item None of these cells are cells at which $P$ and $L$ currently agree.  This eliminates at most $6\epsilon n$ choices.
\item None of the symbols $s_3, s_4,$ or $s_5$ are equal to any of the symbols $\{t_1, \ldots t_a\}$.  This eliminates at most $3a$ choices.
\end{itemize}
This leaves at least
\begin{align*}
n - 4\frac{k}{d}n - 6d n  - 6 \epsilon n - 3a -2
\end{align*}
choices.

Before making our final choice, notice that in the original $\begin{array}{|c|c|}\hline A & B \\ \hline B^T & A^T \\ \hline\end{array}$ form of our latin square $L$, the symbols $s_3$ and $s_5$ have to originally have came from the same quadrant.  This is because none of the cells $(c_1, r_2),$ $(c_4, r_3),$ $(c_1, r_3),$ $(c_4, r_2)$ have been disturbed in prior trades, and the two cells $(c_2, r_4), (c_3, r_2)$ contain the same symbol $s_2$.

Using this observation, choose $s_6$ so that the following properties hold:
\begin{itemize}
\item Choose $s_6$ so that it is in the opposite half from the symbols $s_3, s_5$.  This eliminates at most $\lceil n/2\rceil$ choices.
\item None of the cells containing $s_6$ in columns $c_1, c_4$ or rows $r_2, r_3$, nor the cell containing $s_3$ in column $c_6$, nor the cell containing $s_5$ in column $c_5$, have been used in previous trades.  Because the columns $c_1, c_4$, rows $r_2, r_3$, and symbols $s_3, s_5$  are all not overloaded, this is possible, and eliminates at most $6 d n$ choices.
\item None of these cells are places where $P$ and $L$ agree.  This eliminates at most $6\epsilon n$ choices.
\item The symbol $s_6$ has not been chosen before, nor is it equal to any of the symbols $\{t_1, \ldots t_a\}$.  This eliminates at most $5+a$ choices.
\end{itemize}
This leaves at least 
\begin{align*}
\left\lfloor n/2\right\rfloor  - 6d n  - 6 \epsilon n - a - 5
\end{align*}
choices. 

Using our ``many $2 \times 2$ subsquares'' structure tells us that we have constructed the claimed trade.  Therefore, as long as we can make these choices, we can find one of these two trades.  By looking at all of the choices we make during our proof and choosing the potentially strictest bounds (under certain choices of $d, \epsilon, k, n$) we can see that such trades will exist as long as
\begin{align*}
3 &\leq n - 4\frac{k}{d }n - 6d n  - 6 \epsilon n - 3a, \textrm{ and} \\
12 &\leq n - 12d n -  12\epsilon n - 4a.\\
\end{align*}

\end{proof}
Note that an analogous result holds for exchanging the contents of almost any two cells in a given column, using the exact same proof methods.

The next lemma, built off of Lemma $\ref{lem2}$, is the main tool used in this paper.
\begin{lem}\label{lem3}
As before, let $L$ be one of the $n\times n$ latin squares constructed by Lemma $\ref{lem1}$, and $P$ be an $ \epsilon$-dense partial latin square.  Suppose that we have performed a series of trades on $L$ that have changed the contents of no more than $k n^2$ of $L$'s cells.  (Note that we count the $3n+7$ potentially-disturbed cells from the construction of $L$, when we enumerate these changed cells.)

Fix any cell $(r_1,c_1)$ such that $P(r_1,c_1)$ is filled and does not equal $L(r_1,c_1)$.  Then, there is a trade on $L$ that
\begin{itemize}
\item does not change any cells on which $P$ and $L$ currently agree,
\item changes the contents of at most 69 other cells of $L$, and
\item causes $P$ and $L$ to agree at the cell $(r_1,c_1)$, 
\end{itemize}
whenever we satisfy the bound
\begin{align*}
20 &\leq n - 12 n\sqrt{k} -  12\epsilon n.\\
\end{align*}
\end{lem}
\begin{proof}

Let $d > 0$ be some constant, corresponding to the notion of an ``overloaded'' row we introduced earlier.  

Suppose that $L(r_1, c_1) =s_1 \neq P(r_1, c_1) = s_2$.  Let $r_2$ be a row and $c_2$ be a column such that $L(r_2, c_1) = L(r_1, c_2) = s_2$.  

Our goal in this lemma is to construct a trade that causes $L$ and $P$ to agree at $(r_1, c_1)$, without disturbing any cells at which $P$ and $L$ already agree.  We will do this using four successive applications of Lemma $\ref{lem2}$, one each on row $r_1$, row $r_3$, column $c_1$, and column $c_3$, as illustrated in the picture below:

\begin{center}
\includegraphics[width=5in]{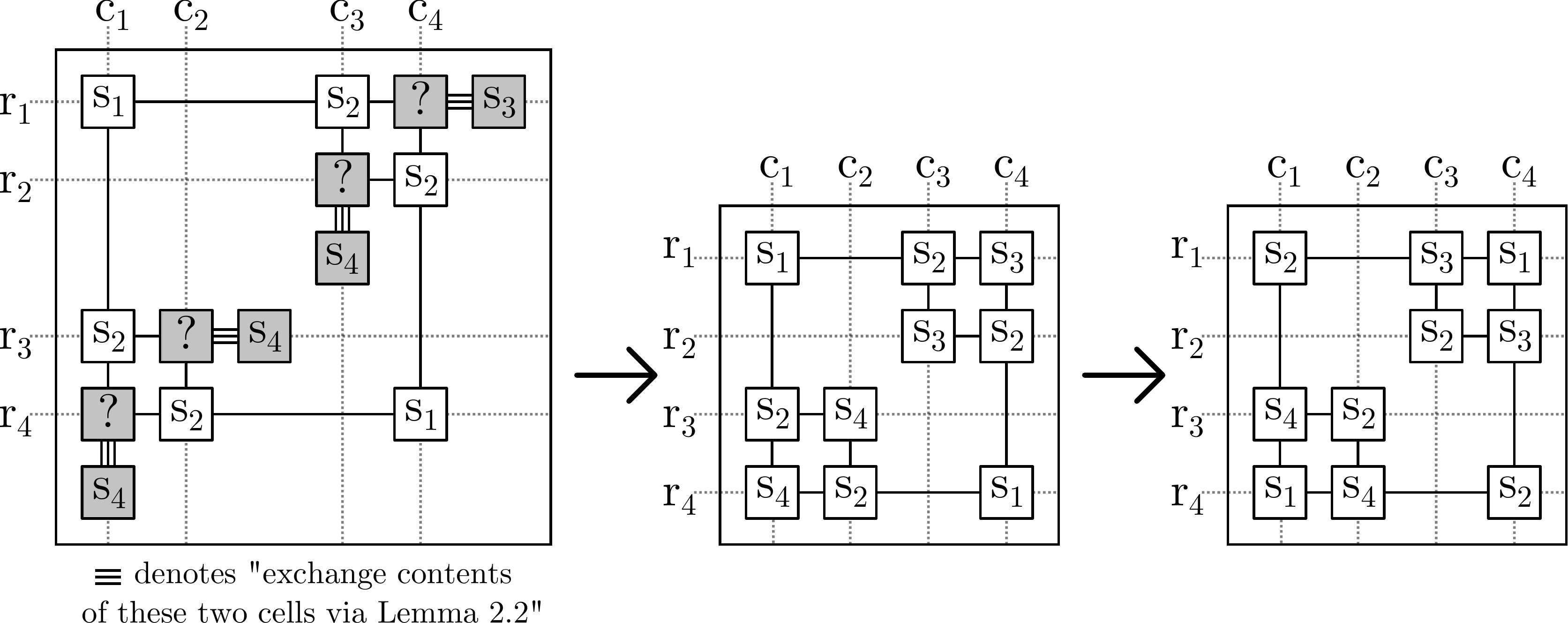}
\end{center}
Assuming we apply Lemma $\ref{lem2}$ as claimed, performing the subsequent trade illustrated in the diagram causes $L$ and $P$ to agree at the cell $(r_1, c_1)$.  Therefore, it suffices to show how we can use Lemma $\ref{lem2}$ as illustrated above. 

 First, we note that in our applications of Lemma $\ref{lem2}$ we will avoid the symbols $\{s_1, s_2\}$, to avoid any conflicts.  Next, notice that picking the cell $(r_4, c_4)$ that contains $s_1$ determines the rows $r_2, r_4$ and the columns $c_2, c_4$. Because of this, we want to choose this cell such that the following properties hold:
\begin{itemize}
\item The cells $(r_4, c_4), (r_4, c_2), (r_2, c_4)$ are not ones at which $P$ and $L$ currently agree.  This eliminates at most $3\epsilon n$ choices
\item The four cells $(r_4, c_1), (r_2, c_3), (r_3, c_2), (r_1, c_4)$ are all valid choices for the first cell to be exchanged in an application of Lemma $\ref{lem2}$. Via Lemma $\ref{lem2}$, this eliminates at most $4\left( 2\frac{k + \frac{48}{n^2}}{d}n + \epsilon n + 2\right)$ choices.  (The $+ \frac{48}{n^2}$ comes from the fact that we are applying Lemma $\ref{lem2}$ four consecutive times, and therefore on the fourth application of our lemma our latin square $L$ may contain up to $kn^2 + 48$ disturbed cells.)  Also, notice that Lemma $\ref{lem2}$ insures that these cells are not ones at which $P$ and $L$ agree.
\end{itemize}
This leaves us with 
\begin{align*}
n - 8\frac{k+ \frac{48}{n^2}}{d} n - 7\epsilon n - 8
\end{align*}
choices for this cell.

Now, choose the symbol $s_3$ such that the following properties hold:
\begin{itemize}
\item The cells containing $s_3$ in row $r_1$ and column $c_3$ are both valid choices for the second cell to be exchanged in an application of Lemma $\ref{lem2}$.  By Lemma $\ref{lem2}$, this eliminates at most $2\left( 4\frac{k+ \frac{48}{n^2}}{d}n + 2d n + 3 \epsilon n + 3\right)$ choices.  Note that in this calculation we have already ensured that these cells are not ones at which $P$ and $L$ agree.
\end{itemize}
This leaves us with 
\begin{align*}
n - 8\frac{k+ \frac{48}{n^2}}{d} n - 4d n - 6\epsilon n - 6
\end{align*}
choices of this symbol. By an identical chain of reasoning, we have precisely the same number of choices for $s_4$.

 Therefore, if we can make the above pair of choices and additionally satisfy the bounds
\begin{align*}
9 &\leq n - 4\frac{k+ \frac{48}{n^2}}{d}n - 6d n  - 6 \epsilon n ,  \\
20 &\leq n - 12d n -  12\epsilon n \\
\end{align*}
required by Lemma $\ref{lem2}$, we can find the requested trades.  Do so one by one, performing one of the $s_3$ trades, then the other, then the corresponding $s_2$-$s_3$ $2 \times 2$ trade created by these two squares, then one of the $s_4$ trades, then the other, and finally the corresponding $s_2$-$s_4$ $2\times 2$ trade.  Because these Lemma $\ref{lem2}$ applications were restricted to not use the symbols $\{s_1, s_2\}$, none of these trades disturb the work done by previous trades, or the $s_1, s_2$ cells in our trade.  Therefore, after performing these trades, we can finally perform the $s_1$-$s_2$ $2 \times 2$ trade created by all of our work, and get the symbol $s_2$ in $(r_1, c_1)$.

Because each application of Lemma $\ref{lem2}$ disturbs the contents of at most 16 cells, and our final trade disturbs 5 other cells apart from $(r_1, c_1)$, we have constructed a trade that disturbs no more than 69 cells other than $(r_1, c_1)$ whenever we can satisfy these five inequalities.

Using basic calculus, it is not too difficult to see that the best choice of $d$ for maximizing the values of $\epsilon, k$ available to us is roughly $\sqrt{k}$.  Therefore, if we let $d = \sqrt{k}$, we can (by comparing our five inequalities) reduce the number of bounds we need to consider down to just one: specifically,
\begin{align*}
20 &\leq n - 12 n\sqrt{k} -  12\epsilon n.\\
\end{align*}
\end{proof}

With these lemmas established, we can now prove the central claim of this paper.
\begin{thm}\label{thm1}
Any $ \epsilon$-dense partial latin square $P$ containing no more than $\delta n^2$ filled cells in total is completable, for $\epsilon < \frac{1}{12}, \delta <\frac{(1 -12\epsilon)^2  }{10409}$.
\end{thm}
\begin{proof}
Take any $ \epsilon$-dense partial latin square $P$ with no more than $\delta n^2$-many filled cells, and let $L$ be a latin square of the same dimension as $P$ as generated by Lemma $\ref{lem1}$.  Cell by cell, select a filled cell $(r_1, c_1)$ of $P$ at which $P(r_1, c_1) \neq L(r_1, c_1)$, and apply Lemma $\ref{lem3}$ to find a trade that disturbs at most 69 other cells and that causes $L$ and $P$ to agree at this cell.  Again, if we can always apply this lemma, iterating this process will convert $L$ into a completion of $P$, and thus prove our theorem.

So: we start with a square in which at most $3n+ 7$ cells were disturbed, and proceed to disturb $69 \delta n^2$ more cells via our repeated applications of Lemma $\ref{lem3}$.  If we want this to be possible, we merely need to choose $n, \delta, \epsilon$ such that the inequality
\begin{align*}
20 &\leq n - 12 \sqrt{69\delta n^2 + 3n + 7} -  12\epsilon n\\
\end{align*}
holds.  For any $\epsilon < \frac{1}{12}$, we can choose any value of
\begin{align*}
 \delta <\frac{(1 -12\epsilon)^2 - \frac{40}{n}(1-12\epsilon) + \frac{400}{n^2} - \frac{432}{n} - \frac{1008}{n^2} }{9936},
\end{align*} 
and our inequality will hold.  For $ \delta < \frac{1}{n}$, our latin square contains $\delta n^2 < n$ symbols, and is therefore completable via a result of Smetianuk.  Otherwise, solve the above inequality for $(1-12\epsilon)^2$ to get
\begin{align*}
  9936 \cdot  \delta  + \frac{40}{n}(1-12\epsilon) + \frac{432}{n} + \frac{608}{n^2}  < (1 -12\epsilon)^2.\\
\end{align*}
Now, if we use our observation that $\delta \geq \frac{1}{n}$, and also the observation that this theorem will only give nontrivial results for values of $n > 10^4$,  we can simplify this to the slightly weaker but much more compact inequality
\begin{align*}
 \delta <\frac{(1 -12\epsilon)^2 }{10409}.
\end{align*} 
\end{proof}
\begin{cor}
Any $\frac{1}{13}$-dense partial latin square containing no more than $5.7 \cdot 10^{-7}$ filled cells is completable.
\end{cor}
\begin{cor}
All $\kern-2pt 9.8 \cdot 10^{-5}$-dense partial latin squares are completable.
\end{cor}
\begin{cor}
All $10^{-4}$-dense partial latin squares are completable, for $n > 1.2 \cdot 10^5.$
\end{cor}

\begin{proof}
The first corollary is immediate from Theorem $\ref{thm1}$.  For the second and third: if we set $\epsilon = \delta$, we are simply dealing with a $ \epsilon$-dense partial latin square.  Whenever $\delta = \epsilon < \frac{2}{n}$, a $ \epsilon$-dense partial latin square is a latin square with no more than 1 entry in any row, column, and symbol, and is completable (either exactly one entry is used in every row, column, and symbol, in which case this is a transversal and is easily completable; otherwise, we have a latin square with less filled cells than its order, which we know is completable due to a result of Smetianuk.)  Otherwise, if $\delta = \epsilon > \frac{2}{n}$, we can use the first, longer inequality in Theorem $\ref{thm1}$ to see the other two inequalities.
\end{proof}

\section{Acknowledgements}\footnotesize
We thank Richard Wilson for his advice throughout the research process, as well as Peter Dukes and Esther Lamken for several remarkably useful discussions.  As well, thanks go to Andre Arslan, Alan Talmage and Sachi Hashimoto for detecting errors and helping develop part of Lemma $\ref{lem1}$.  Finally, thanks go to the referees for their thoughtful comments and assistance.
\normalsize
\bibliographystyle{plain}	
\bibliography{myrefs}
\end{document}